\documentclass[11pt,a4paper]{amsart}
\newcommand*\patchAmsMathEnvironmentForLineno[1]{%
  \expandafter\let\csname old#1\expandafter\endcsname\csname #1\endcsname
  \expandafter\let\csname oldend#1\expandafter\endcsname\csname end#1\endcsname
  \renewenvironment{#1}%
     {\linenomath\csname old#1\endcsname}%
     {\csname oldend#1\endcsname\endlinenomath}}%
\newcommand*\patchBothAmsMathEnvironmentsForLineno[1]{%
  \patchAmsMathEnvironmentForLineno{#1}%
  \patchAmsMathEnvironmentForLineno{#1*}}%
\AtBeginDocument{%
\patchBothAmsMathEnvironmentsForLineno{equation}%
\patchBothAmsMathEnvironmentsForLineno{align}%
\patchBothAmsMathEnvironmentsForLineno{flalign}%
\patchBothAmsMathEnvironmentsForLineno{alignat}%
\patchBothAmsMathEnvironmentsForLineno{gather}%
\patchBothAmsMathEnvironmentsForLineno{multline}%
}

\usepackage{amsthm, amsfonts, amssymb, amsmath, latexsym, enumerate, array}
\usepackage[all]{xy}
\usepackage[pdftex]{graphicx}

\usepackage[latin1]{inputenc}
\usepackage{hyperref,cite,mdwlist,mathrsfs}

\usepackage[left=3.3cm,top=2.7cm,right=3.3cm]{geometry}

\newtheorem{thm}{Theorem}[section] 
 \newtheorem{pro}[thm]{Proposition}
\newtheorem{cor}[thm]{Corollary}

\newtheorem*{thm*}{Theorem} \newtheorem*{cnj*}{Conjecture}

\theoremstyle{definition} 
 \newtheorem{defi}[thm]{Definition}
\newtheorem*{rem}{Remark} 

\newtheorem*{conj*}{Conjecture}

 \newcommand{\C}{\mathbb C}

 \newcommand{\p}{\mathbb P}

\SelectTips{cm}{} \newdir{ >}{{}*!/-5pt/@{>}}
\numberwithin{equation}{section}

\allowdisplaybreaks[1]
\begin{document}


\title{A Poncelet theorem for lines}

\author{Jean Vall\`es}


\thanks{Author partially supported by ANR-09-JCJC-0097-0 INTERLOW and ANR GEOLMI}

\begin{abstract}
Our aim is to prove  a Poncelet type theorem for a line configuration  on the complex projective 
plane $\p^2$.
More precisely, we say that a polygon with $2n$ sides joining $2n$ vertices $A_1, A_2, \cdots , A_{2n}$ 
is well inscribed in a configuration $\mathcal{L}_n$ of $n$
lines if each line of the configuration contains exactly two points among $A_1, A_2, \cdots , A_{2n}$.
Then we prove : 

\smallskip

\noindent \textbf{Theorem}
\textit{Let  $\mathcal{L}_n$ be a configuration of  $n$ lines  and   $D$ a smooth conic in  $\p^2$. 
If it exists one polygon with  $2n$ sides well inscribed in $\mathcal{L}_n$ and circumscribed around   $D$
then there are  infinitely many  such polygons. In particular a general point in  $\mathcal{L}_n$ is a vertex of such a polygon.}

\smallskip

This result was probably known by  Poncelet or Darboux but we did not find a 
similar statement in their publications. Anyway, even if it was, we would like to propose an elementary proof based on Fr\'egier's involution. 
We begin by recalling some facts about these involutions. Then we explore the 
following question : When does the product of involutions correspond to an  involution?
We give a partial answer in proposition \ref{alignes}. 
This question leads also to Pascal theorem, to its dual version proved by Brianchon, and to its generalization 
proved by M\"obius (see \cite{A}, thm.1 and \cite{Mo}, page 219).
In the last section, using the  Fr\'egier's involutions and the projective duality we prove the main theorem quoted above.

\end{abstract}

\maketitle
\today

\section{Introduction}

Let $\mathcal{L}_n$ be a configuration of  $n$ lines $L_1, \cdots, L_n$ in the complex projective plane $\p^2$
and $D$ be a smooth conic 
in the same plane. We assume that $\mathcal{L}_n\cap D$ consists in $2n$ distinct points. From a point $A_1$ on $L_1$ (not being 
on the other lines neither on $D$)
we  draw a tangent line to $D$. This line cuts $L_2$ in a point $A_2$. In the same way we define successively 
$$A_3\in L_3, \cdots, A_n\in L_n, A_{n+1}\in L_{n-1}, \cdots , A_{2n}\in L_2.$$
The second tangent line to $D$ from $A_{2n}$ meets $L_1$ in $A_{2n+1}$. 
Now two situations can occur : $A_{2n+1}=A_1$  or $A_{2n+1}\neq A_1$. This second case is  clearly  the general case.

\smallskip

If $A_{2n+1}=A_1$  the polygon $(A_1A_2)\cup \cdots \cup (A_{2n}A_{2n+1})$ is simultaneously inscribed in $\mathcal{L}_n$ and circumscribed around $D$.
Our aim is to prove that in this case, there are  infinitely many polygons with $2n$ sides  simultaneously inscribed in $\mathcal{L}_n$ and circumscribed around $D$
 (see theorem \ref{produitdedroites}). It means that the existence of such polygons does
 not depend on 
the initial point but only on $D$ and $\mathcal{L}_n$.

\smallskip

On the contrary, if $A_{2n+1}\neq A_1$  the polygon will 
never close after $2n$ steps for any initial point on $\mathcal{L}_n$.

\smallskip

 This kind of result is called a porism and it is of the same nature than Steiner porism 
(see \cite{BB}, thm.7.3) for circles tangent to two given circles or Poncelet porism
for two conics (see \cite{GH}, \cite{BKOR} or \cite{Po}).

\smallskip

To prove this porism we \textit{dualize} the situation and we consider the dual polygon inscribed in $D$. To be more explicit, let us assume that $A_{2n+1}=A_1$. 
Then the situation can be dualized in the following way. Any line $L_i$ in the configuration is the polar line of a point $c_i$ (its pole).
Any tangent line $(A_iA_{i+1})$  is the polar line of a point $x_{i,i+1} \in D$  for $1\le i\le 2n$.
For $1\le j\le n$, the lines $(x_{j-1,j}x_{j,j+1})$ and $(x_{j+n-1,j+n}x_{j+n,j+n+1})$ meet in $c_i$. By this way we obtain  an inscribed polygon 
$(x_{1,2}x_{2,3}) \cup \cdots \cup (x_{2n,2n+1}x_{1,2})$ with $2n$ sides passing through $n$ fixed points $c_1, \cdots, c_n$
(this inscribed polygon corresponds to the choice of a $2n$-cycle among the permutations of $2n$ points).

 \begin{figure}[h!]
    \centering
    \includegraphics[height=8.5cm]{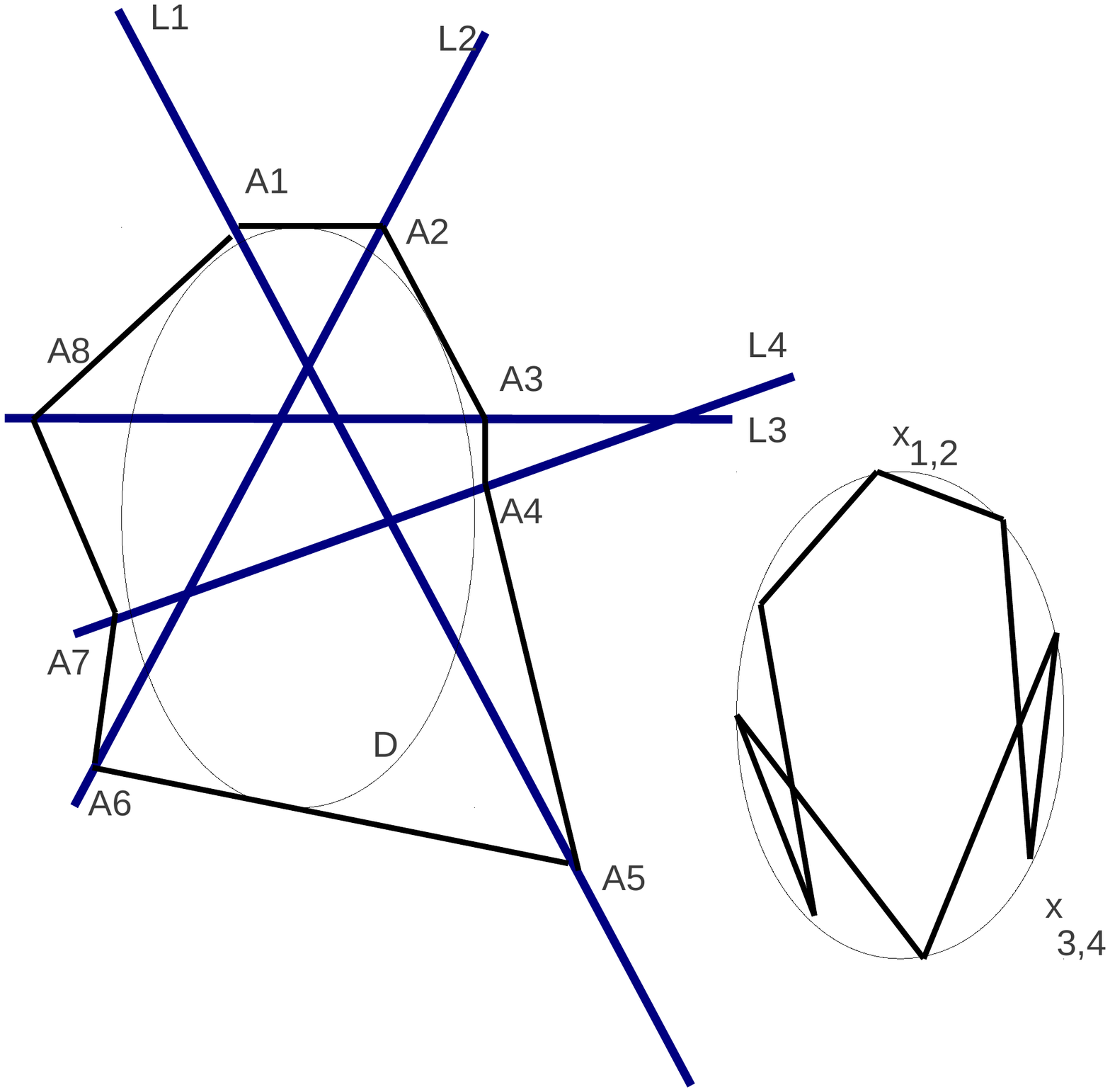}
    \caption{Octagon inscribed in $D$ and its dual inscribed in $4$ lines and circumscribed around $D$. }
  \end{figure}
This inscribed polygon  leads to study Fr\'egier's involutions that are particular  automorphisms of $D$. 
Indeed we verify that the porism (seetheorem \ref{produitdedroites}) is true when  the product of the $n$ Fr\'egier's involutions
giving the incribed polygon is also a Fr\'egier's involution. 

\smallskip

So first of all we  recall the definition 
and the  main properties of these automorphisms. 
Then we study the product of involutions. More precisely we wonder when a product of $n$ involutions is still an involution.\\
For $n=2$, it is very well known (see proposition \ref{two}).\\
For $n=3$, we prove that the product is an involution if and only if the centers are aligned (see proposition \ref{uvw2}). It is an other way to set out the 
so-called Pascal theorem\footnote{Another way but probably not a new way. According to its  simplicity this argument is certainly already written somewhere.}.\\
For $n\ge 4$, we propose a new proof of a generalization of Pascal theorem due to M\"obius (see theorem \ref{mob}). We prove also that the product of $2n+1$ involutions 
is an involution when their centers are aligned (see proposition \ref{alignes}).

\smallskip

The projective duality give us the dual versions of all these results, like Brianchon theorem for instance (see theorem \ref{dmob} and proposition \ref{dalignes}).

\smallskip

In the last section we prove our Poncelet type theorem for lines (see theorem \ref{produitdedroites}) and we conclude by an explicit computation in the case of two lines 
simultaneously inscribed and circumscribed around a smooth conic.

\section{Fr\'egier's involutions}
The group $ \mathrm{PGL}(2,\C)$ acts on $\p^1=\p^1(\C)$ in the usual way. 
\begin{defi}
 An element  $g\in \mathrm{PGL}(2,\C)$ which is not the identity $I$ on $\p^1$ is called an involution
 if    $g^2= I$.
\end{defi}

Considering  $ \mathrm{GL}(2,\C)$ as the space of $2\times 2$ invertible matrices it is clear that any $g\in   \mathrm{PGL}(2,\C)$ 
 has at most two fixed points ($g$ has three fixed points if and only if
$g= I$). 
Moreover, when $g$ is an involution it is easy to verify that 
it has exactly  two fixed points and that it is determined  by the data of these fixed points. 

\smallskip

The Veronese embedding  $\p^1 \stackrel{v_2}\hookrightarrow \p^2$ induces an embedding for groups  $\mathrm{PGL}(2,\C) \subset \mathrm{PGL}(3,\C)$. 
 Let  $g\in \mathrm{PGL}(2,\C)$ be an involution on $\p^1$. 
The  corresponding transformation $g$ in $\p^2$ has  two fixed points on
the smooth conic $D=v_2(\p^1)$: the images of the fixed points on $\p^1$. \\
Let us consider now the intersection point $x\notin D$ of the two  tangent lines in the fixed points of $g$.
 A general line through $x$ cuts $D$ in two points. The map exchanging 
these points is an involution on $D$ (i.e. on $\p^1$). Its fixed points are 
the  intersection points of $D$ with the polar line of $x$. Such an involution  is called \textit{Fr\'egier involution} on $D$ with center $x$. 

Since an involution is determined by its fixed points, this involution is  $g$.
We have verified that:
\begin{pro}
 Any involution on $D\subset \p^2$ is a Fr\'egier's involution.
\end{pro}
\subsection{Product of two and  three involutions}
Rigorously speaking we should not write product but composition of involutions. Anyway, from now, since for matrices it becomes a product we will  
write product in any case. Moreover we will denote $uv$ the product (i.e the composition) of two involutions $u$ and $v$  and $u^n$ the product (i.e the composition)
 of $u$ with itself $n$  times.
 The following proposition is  
classical and easy to prove.
\begin{pro}
\label{two}
Let $u$ and $v$ be two involutions with distincts fixed points. Then,
$ uv$ is involutive if and only if  the two fixed points of $u$
and the two fixed points of $v$  form an harmonic divison of $D$.
\end{pro}
For three  involutions
we recognize Pascal theorem.
\begin{pro}
\label{uvw2} Let $u$, $v$, and $w$ be three involutions with distincts fixed points and let $x_u$,
$x_v$ and
$x_w$ be their respective centers. Then,
$$ uvw\,\,\textrm{is involutive}\,\,\Leftrightarrow  x_u, x_v\,\,\textrm{and}\,\, x_w
\,\,\textrm{are aligned}.$$
\end{pro}
\begin{proof}\footnote{This  proposition, its proof and its corollary appear in a book project with Giorgio Ottaviani (see \cite{wykno}).} 
Assume that the three centers are aligned on a line $L$.
The line $L$ is not tangent to $D$ because the three involutions do not have a common 
fixed point. Then  let $\{x,y\}= L\cap D$. We verify that $uvw(x)=y$ and $uvw(x)=y$. The automorphism 
$uvw$ has at least one fixed point
$z\in D$. Now the three points $x,y$ and $z$ are fixed points for
$(uvw)^2$. 
 \begin{figure}[h!]
    \centering
    \includegraphics[height=6.5cm]{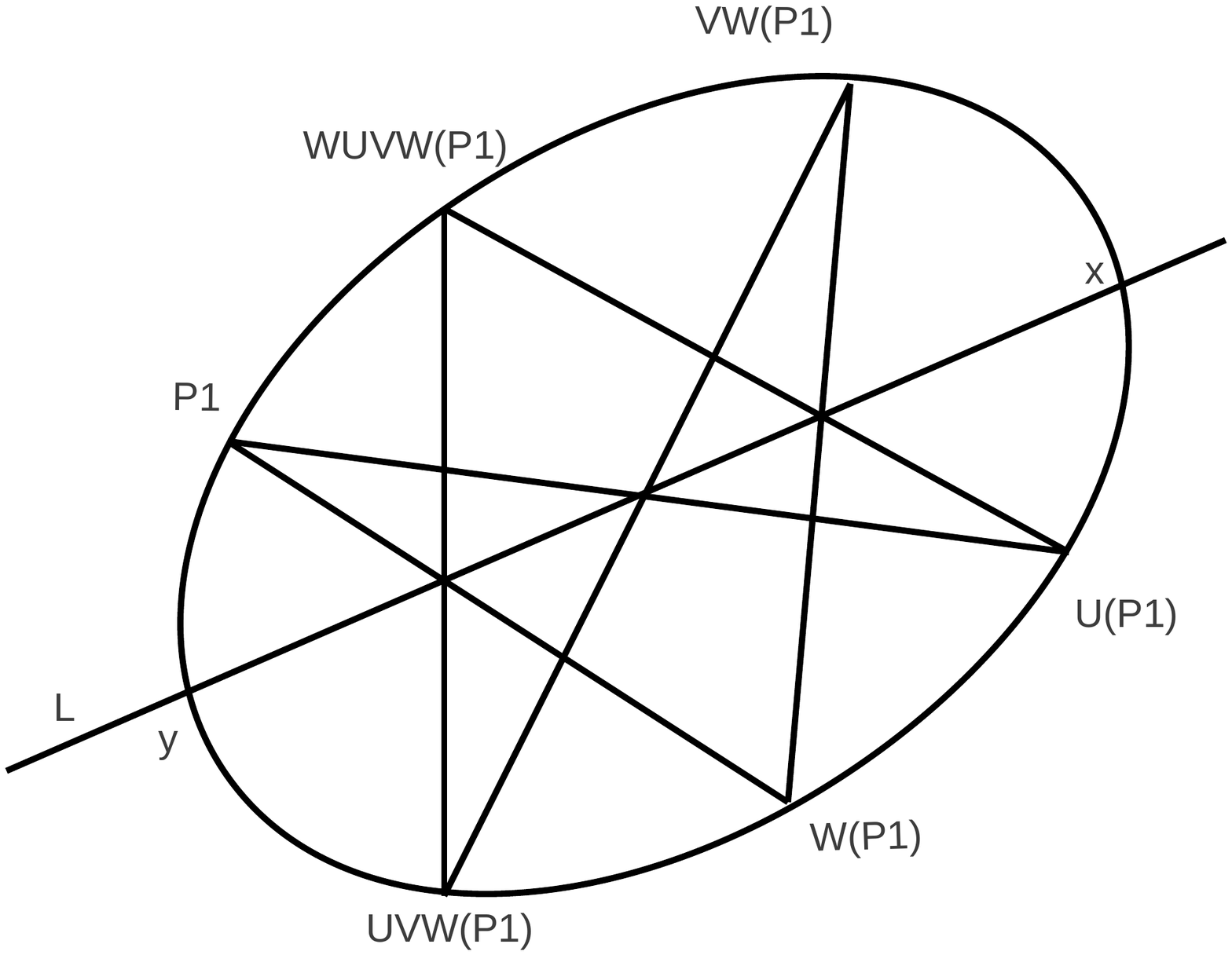}
    \caption{Three involutions with aligned centers. }
  \end{figure}
It means that $ uvw$ is an involution on $D$.

\smallskip

Conversely, assume that $uvw$ is involutive. Let $x\in D$ such that  $v(x)=w(x)\neq x$. Call $L$ the line joining $x$ to $v(x)$, i.e passing trough $x_{v}$ and $x_{w}$. From
$v(x)=w(x)\neq x$ and the assumption $uvw=wvu$, we find four fixed points of $wv$:
$x, v(x), u(x), uv(x)$. 
The automorphism $wv$ has at most two distinct fixed points. The first two are
distinct. Consider the third one $u(x)$.
If $u(x)=v(x)$ we have $u(x)=v(x)=w(x)$ and we have finished, hence $u(x)=x$.
Consider the fourth one $uv(x)$. If $uv(x)=x $ we have again $u(x)=v(x)=w(x)$ and we
have finished, hence $uv(x)=v(x)$.
It follows that $x, v(x)$ are fixed points of both $u$ and $uvw$, that is a
contradiction since an involution is uniquely determined by its fixed points.
\end{proof}
\begin{cor}[Pascal's theorem]
Let $p_1, p_2, p_3, q_3, q_2, q_1$ be six (ordered) points on a smooth conic
$D$. Let $x_{ij}, i<j$ the 
 intersection point of the two lines joining $p_i$ to $q_j$ and $p_j$ to $q_i$.
Then the three points
$x_{12}, x_{13}$ and $x_{23}$ are aligned.
\end{cor}
\begin{proof}
We denote by $u$ the involution defined by $x_{13}$, $v$ the
one defined by $x_{23}$ and
$w$ the last one defined by $x_{12}$. Then by following lines we verify that
$$(uvw)(p_1)=q_1, (uvw)(q_1)=p_1.$$
Let $z$ be a fixed point of $uvw$. Then $z,p_1,q_1$ are fixed points of $ (uvw)^2$.
Since an element of $\textrm{PGL}(2,\C)$ that has more than three fixed
points is the identity, we have
proved that $ uvw$ is an involution. The result now follows from proposition \ref{uvw2}.
\end{proof}

\begin{rem}
\label{coropascal}
The center of the product of three involutions $u_1, u_2$ and $u_3$ with aligned centers (on a line $L$) 
belongs also to $L$. Indeed let us define $v=u_1 u_2 u_3$; since the centers are aligned $v$ is also an ivolution.  Then 
$u_1 u_2 u_3v=I$. It  implies $u_1=u_2u_3 v$ i.e that the product of $u_2,u_3$ and $v$ is an involution. According to  theorem 
\ref{uvw2}  their 
centers are aligned.
\end{rem}

\subsection{Product of $n\ge 4$ involutions and M\"obius theorem}
We show that the product of an odd number of involutions with aligned centers is still an involution.
\begin{rem}
We cannot expect an equivalence as it was in proposition \ref{uvw2} but only an implication.
Indeed let us give a product of five  involutions with three centers on a line $L$ and two centers on another line $D$ that is also an involution.
So, let us consider a line $L$ and three points $x_1,x_2,x_3$ on it. We associate three  involutions $u_1,u_2, u_3$ to these centers.
The product $w=u_1 u_2 u_3$ is also an involution by proposition \ref{uvw2} . Moreover, according to  remark \ref{coropascal} its center  $x$ belongs to $L$.
 Now we introduce another line $D$ passing through $x$,
and two others points $x_4,x_5$ on $D$. Let us call $u_4, u_5$ the associated involutions. Then 
$$ u_1 u_2 u_3 u_4  u_5=w u_3 u_4,$$
and $w u_3 u_4$ is an involution because their  centers $x,x_4,x_5$ are aligned.
\end{rem}
\begin{pro}
\label{alignes}
 Let  $u_1, \cdots, u_{2n+1}$ be  involutions on   $D\subset \p^2$ with respective centers 
  $c_1,\cdots, c_{2n+1}$. If $c_1,\cdots, c_{2n+1}$ are  aligned then 
  $\prod_{i=1}^{2n+1}u_{i}$ is an involution. 
\end{pro}
\begin{proof}
Let $v=\prod_{i=1}^{2n+1}u_{i} $. We recall that the fixed points of $v$ are also fixed points of $v^2$.
The automorphism  $v$ possesses at least one fixed point $x$.  
Let $L$ be the line of centers and  $L\cap D =\lbrace y,z \rbrace$. These two points are exchanged by 
$v$ and then are not fixed points of $v$. The points $x,y,z$ are fixed points of $v^2$, then $v^2=I$. 
\end{proof}
\begin{rem}
 The proposition is not valid for an even number of involutions. Indeed since two points are always aligned it is clearly not valid for two involutions.
Consider now three involutions with aligned centers $u_1,u_2,u_3$. Call $w$ their product. Let  $u_4$ be another involution with its center on the previous 
line of centers. Then the product $w u_4$ is an involution if and only if  the four fixed points form an  harmonic division. But if we move the center of $u_4$ on $L$
 the cross-ratio is changing. So for the general point on $L$ the product will not be involutive.
\end{rem}
In general when the product of involutions is an involution we are not able to say something 
pertinent about the position of their centers. But, when the product of $n$ involutions  is still an involution and at least $n-1$ centers are aligned, 
M\"obius proved that 
all the centers are aligned (see \cite{Mo}, page 219).  
We prove again this theorem in the terminology of Fr\'egier's involution. The formulation below  is the one given in  Adler's article (see \cite{A}, thm. 1).
\begin{thm}[M\"obius theorem]
\label{mob}
Let  $x_1,y_1, \cdots, x_n, y_n$ be points on a smooth conic. Consider the intersection points 
$a_{j}=(x_jx_{j+1})\cap (y_jy_{j+1})$,  $j=1,\cdots,n-1$ and 
$$ a_n= \left \{
             \begin{array}{ccc}
              (x_ny_{1})\cap (y_nx_{1}) &  \mathrm{if}   &  n=2m+1          \\
             (x_nx_{1})\cap (y_ny_{1}) &  \mathrm{if}   &  n=2m.
\end{array}
      \right.$$
If all of these points except possibly one are collinear then the same is true for the remaining point.
\end{thm}
\begin{proof}
As we have seen before (in proposition \ref{uvw2})   it is true for three involutions
since  the product of three involutions is an involution if and only if the centers are aligned. Moreover as we said in remark \ref{coropascal}
the center of the product 
is also aligned with the  three others.

 \begin{figure}[h!]
    \centering
    \includegraphics[height=8.5cm]{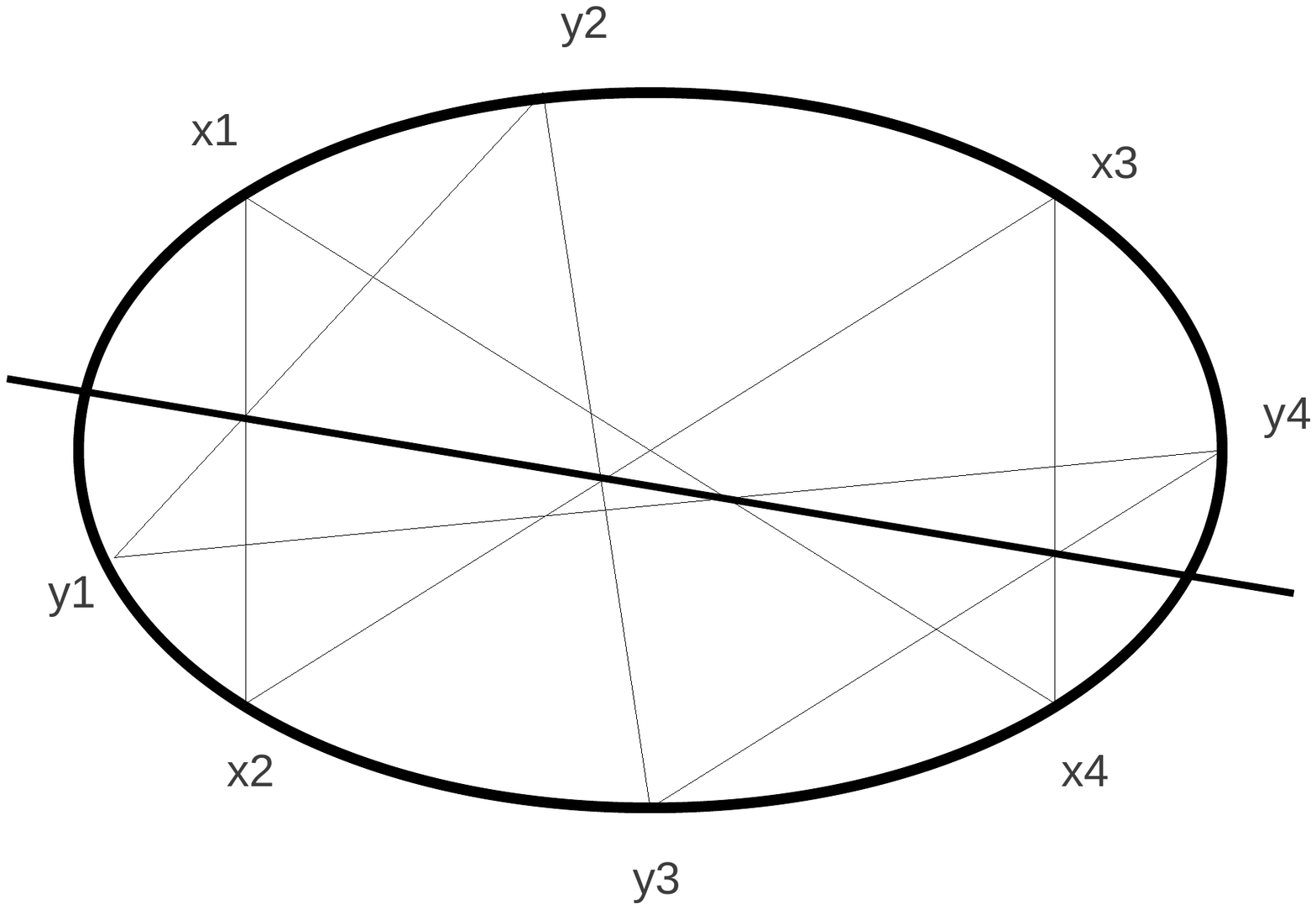}
    \caption{Four aligned centers. }
  \end{figure}

We need only   to prove the result for four involutions. Indeed let us verify that we can reduce the general case ($n\ge 4$) to three or four involutions. 
Consider $n> 4$ involutions.
Any group of three terms among $u_{i-1}, u_i, u_{i+1}$ for $i=2, \cdots, n-2$ gives a new involution with a center on the same line.
So when $n$ is odd it remains three involutions after reductions $u u_{n-1} u_{n}$. Since the product is  an involution the centers are aligned.
But the centers of $u$ and $u_{n-1}$ are already on the line $L$. Then the center or $u_n$ is also on $L$.
When $n$ is even it remains four involutions after reduction  $u u_{n-2} u_{n-1} u_{n}$ with the first three centers on $L$.

\smallskip

So let us consider the case $u_1 u_{2} u_{3} u_{4}=v$ with $v$ an involution. We have $u_1 v=u_2 u_3 u_4$.
The center $x_4$ belongs to $(x_2x_3)$ if and only if $u_2 u_3 u_4$ is an involution, i.e. if and only if  $u_1v$ is an involution. 
To prove it  let us show that $u_1 v=v u_1$. 
Since $x_1,x_2,x_3$ are aligned the product $u_1u_2u_3$ is involutive, so $u_1u_2u_3=u_3u_2u_1$.
Then $u_2u_3=u_1u_3u_2u_1$ and $u_1v=(u_2u_3)u_4=(u_1u_3u_2u_1)u_4$. Since
$u_3u_2u_1=u_1u_2u_3$ we obtain 
$$u_1v=(u_1u_3u_2u_1)u_4=u_1(u_1u_2u_3)u_4=u_1v.$$
\end{proof}

\subsection{Projective duality and involutions}
All the results obtained above can be \textit{dualized} by considering polar lines of  points and poles of lines with respect to the smooth conic 
$D$. By this way any inscribed polygon into $D$   induces a circumscribed polygon (with the same number of sides) around $D$.
 Even if M\"obius theorem was certainly  \textit{dualized} by M\"obius himself, 
we write one more time this dual version below.

\begin{thm}[dual M\"obius]
\label{dmob}
Let $L_1\cup \cdots \cup L_{2n}$ be a polygon tangent circumscribed to the smooth conic $D$.
If the diagonals joining  $L_i\cap L_{i+1}$ to $L_{i+n}\cap L_{i+1+n}$  for $i=1, \cdots, n-1$ are concurrent lines then 
 the same is true for the remaining diagonal joining $L_{2n}\cap L_{1}$ to $L_{n}\cap L_{n+1}$.
\end{thm}
\begin{rem}
For $n=3$ it is Brianchon theorem.
\end{rem}
Proposition \ref{alignes} implies that one can construct an inscribed polygon in $D$ when  $n$ aligned points $c_1,\cdots, c_n$ (not on $D$) are given. 
Indeed let $x$ be a point on $D$ and let us take successively its  images
by the involutions $u_i$ associated to the centers $c_i$: 
$$x, u_1(x), (u_2 u_1)(x) , \cdots , y=(u_{2n+1} \cdots  u_1)(x). $$
Then let us take successively the  images of $y$ by the involutions $u_i$:
$$u_1(y), (u_2 u_1)(y) , \cdots , (u_{2n+1} \cdots  u_1)(y). $$
Since the product is involutive the process stops and  we have:
 $$(u_{2n+1} \cdots  u_1)(y)=(u_{2n+1} \cdots  u_1)^2(x)=x.$$
In other words,  from a general  point on $D$ we can draw by this method an inscribed polygon with $4n+2$ sides. 
Dualizing this statement, we verify the following proposition: 
\begin{pro}[dual version of proposition \ref{alignes}]
\label{dalignes}
Let us consider $2n+1$ concurrent lines $L_i$ meeting a smooth conic $D$ in  $4n+2$ distinct points. Then take any point $P_1$
on $L_1$ and draw a tangent to $D$ from this point. This tangent cuts $L_2$ in one point $P_2$. Let us draw successively $P_i\in L_i$ for $1\le i\le 2n+1$
and $P_{2n+1+j}\in L_j$ for $1\le j\le 2n+1$. 
 Then the line 
$(P_1P_{4n+2})$ is tangent to $D^{\vee}$.
\end{pro}

\section{A Poncelet  theorem for lines}
The following  theorem is not a consequence nor of the well known 
Poncelet closure theorem (except when the configuration consists in two lines) neither of Darboux theorem (the configuration is not 
a Poncelet curve associated to $D$ and described in \cite{Tr}).
We say that a polygon with $2n$ sides joining $2n$ vertices 
is well inscribed in a configuration $\mathcal{L}_n$ of $n$
lines when  each line of the configuration contains exactly two  vertices.
\begin{thm}
\label{produitdedroites}
Let  $\mathcal{L}_n$ be a configuration of  $n$ lines  and   $D$ a smooth conic in  $\p^2$. 
If it exists a polygon with  $2n$ sides well inscribed into $\mathcal{L}_n$ and circumscribed around   $D$
 then there are  infinitely many  such polygons. In particular a general point in  $\mathcal{L}_n$ is a vertex of such a polygon.
\end{thm}

\begin{proof}
 The given   polygon of $2n$ sides  well inscribed into  $\mathcal{L}_n$ and circumscribed 
to  $D$ 
corresponds by  duality (polarity) to an inscribed polygon into $D$. 
It gives  $2n$ points on  $D$ linked by  
$2n$ lines that are the polar lines of the considered $2n$ vertices in  $\mathcal{L}_n$. These $2n$ lines meet two by two in $n$ points 
$L_1^{\vee}, \cdots, L_n^{\vee}$ (poles of the  $n$ lines of the configuration). 

\smallskip

Let us show that the product $v=(u_n\cdots u_1)$ of the $n$ involutions $u_1, \cdots , u_n$ with respective centers $L_1^{\vee}, \cdots, L_n^{\vee}$ is involutive.
Let $x_1$ be an intersection point of $L_1\cap D$ and $x_2=v(x_1)$.
Following the sides of this inscribed polygon,  we have 
$v^2(x_{1})=x_{2}$. We have, in the same way, 
$v^2(x_{2})=x_{2}$. Since the inscribed polygon has $2n$ vertices and not only $n$, these two fixed points of the automorphism 
$v^2$ do not co\"{\i}ncide; indeed they are  exchanged by  $v$. Let  $x$ be a fixed point of the product   $v$.
This point  $x$ is also a fixed point of    
$v^2$. Since  $x_{1}$ and  $x_{2}$ are exchanged by   $v$ they do not 
co\"{\i}ncide with   $x$. It implies that  
$v^2$ has three fixed points, i.e. 
$v^2= I$.

\smallskip

Then, a polygon constructed  from a general point $p\in D$ by joining the vertices  
$$\lbrace p,u_1(p),(u_2  u_1)(p),\cdots , (u_{n}   \cdots    u_1)(p), \cdots ,  
(u_{n-1}  \cdots   u_1  u_{n}   \cdots   u_1)(p) \rbrace $$  is inscribed 
in  $D$ and corresponds by duality to an inscribed polygon in $\mathcal{L}_n$  circumscribed around $D$.
\end{proof}
\subsection{Poncelet theorem for singular conics}
Assume that $u_{2i}=u$ and $u_{2i+1}=v$ and let $x$ and $y$ be their respective centers.
We can consider the two polar lines $L_x=x^{\vee}$ and $L_y=y^{\vee}$. Then theorem \ref{produitdedroites}  (with $u_{2i}=u$ and $u_{2i+1}=v$) implies 
 that  $(uv)^n=I$ if and only if there exists a polygon with $2n$ sides inscribed in $L_x\cup L_y$ and circumscribed around $D$.
In that case since an union of two lines is a conic it is a consequence of Poncelet theorem (see \cite{Va}, thm.2.2).

\begin{figure}[h!]
    \centering
    \includegraphics[height=4.5cm]{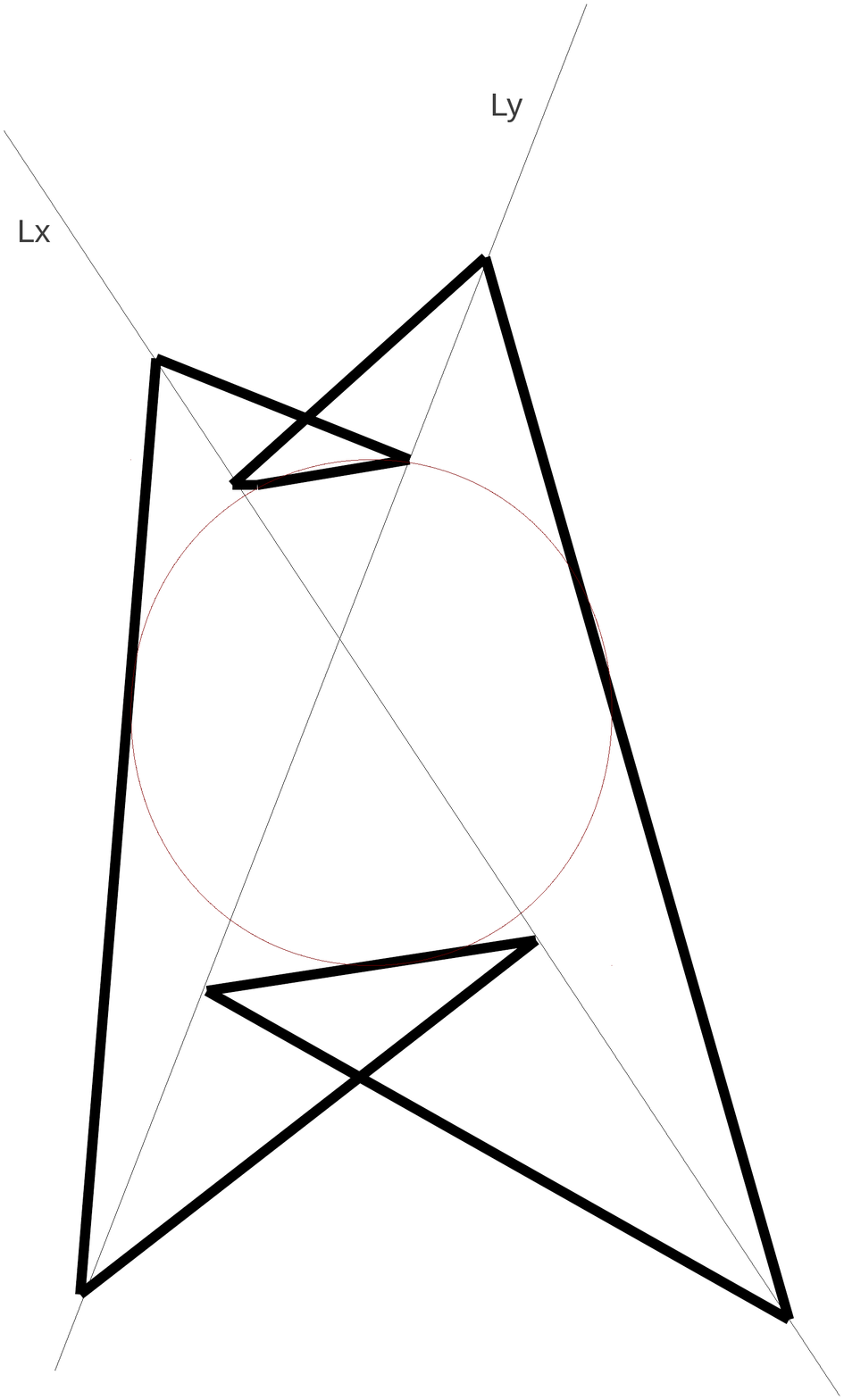}
    \caption{Octagon inscribed in $L_x\cup L_y$. }
  \end{figure}
This situation can be described in an elementary way. Since $\mathrm{PGL(2,\C)}$ acts transitively on triplets of points in $\p^1$
we can choose three among the four fixed points of $u$ and $v$. Then we will obtain a good matrix description. 
 Let us introduce first 
 a family of polynomials\footnote{Quite similar than Fibonacci polynomials.} on the affine line: 

\smallskip

$P_0(x) =1, P_1(x)=x$ and for $n\ge 2$, $P_n(x)=xP_{n-1}(x)-P_{n-2}(x)$. 

\smallskip
 
We can give now  a simple characterization for an union of two lines to be Poncelet associated to a smooth conic.
\begin{pro} Let  $\{ 1,-1\}$ be the fixed points of $u$ and $\{0,2/x\}$  be the fixed points of   $v$. Then,
$$ (uv)^n= I \Leftrightarrow P_{n-1}(x)=0\,\, \mathrm{and} \,\, P_{n-2}(x)\neq 0.$$
\end{pro}
\begin{proof}
Two representatives matrices of   $u$ and   $v$ are  
$$M_u=\left( 
\begin{array}{cc}
 0 & 1\\
 1 & 0 
\end{array}
\right) \,\, \mathrm{and}\,\,
M_v=\left( 
\begin{array}{cc}
 1 & 0\\
 x & -1
\end{array}
\right).$$
The product is given by the matrix   
$$ (M_uM_v)^{n}=\left( 
\begin{array}{cc}
 xP_{n-1}(x)-P_{n-2}(x) & -P_{n-1}(x)\\
 P_{n-1}(x) & -P_{n-2}(x)
\end{array}
\right)   $$ and this matrix is a multiple of the matrix  identity if and only if  $$P_{n-1}(x)=0.$$
\end{proof}


\begin{thebibliography}{OV-Wykno}
\bibitem{A} Adler, V.E. : Some incidence theorems and integrable discrete equations. Discrete Comput. Geom. 36 (2006), no. 3, 489--498,
\bibitem{BB} Barth, W., Bauer, Th. : Poncelet theorems. Expo. Math.
14-2 (1996),
125-144.
\bibitem{BKOR}  Bos, H.J.M., Kers, C., Oort, F., Raven, D.W. : Poncelet's closure theorem, its history, its modern formulation, a
comparison of its modern
proof with those by Poncelet and Jacobi, and some mathematical remarks
inspired by these early proofs. 
Expo. Math. 5 (1987
), 289-364. 
\bibitem{Da} Darboux, G. : Le\c{c}ons sur les syst\`emes orthogonaux et les coordonn\'ees curvilignes. Principes de g\'eom\'etrie analytique.  
 Les Grands Classiques Gauthier-Villars. \'Editions Jacques Gabay, Sceaux, 1993.
\bibitem{GH} Griffiths, P., Harris, J. : On Cayley's explicit solution to Poncelet's porism.  Enseign. Math. (2)  24  (1978), no. 1-2, 31--40.
\bibitem{Mo} M\"obius, F. A. : Verallgemeinerung des Pascal'schen Theorems das in einen Kegelschnit beschriebene Sechseck betreffend. J. Reine Angew. Math.
no. 36 (1848), 216--220.
\bibitem{wykno} Ottaviani, G., Vall\`es, J. : Moduli of vector bundles and group action. Wykno (Poland) Autumn school in Algebraic Geometry, 2001.
In author's web pages.
\bibitem{Po} Poncelet, J. V. : Trait\'e des propri\'et\'es projectives des figures, Tome I et  II. 
Reprint of the second (1866) edition. Les Grands Classiques Gauthier-Villars.  \'Editions Jacques Gabay, Sceaux, 1995.
\bibitem{Tr} Trautmann, G. : Poncelet curves and associated theta characteristics, Expo. Math. (6)
(1988), 29--64.
\bibitem{Va} Vall\`es, J. :  Porisme de Poncelet et coniques de saut.   C. R. Acad. Sci. Paris S\'er. I Math.  333  (2001),  no. 6, 567--570.
\end{thebibliography}
\end{document}